\documentclass[review]{elsarticle}

\usepackage{lineno,hyperref,color}
\usepackage{amsmath,amsthm,amsfonts,amssymb,graphicx,graphics,dsfont,float,color,bm}
\modulolinenumbers[5]
\newcommand{\argmax}{\mathop{\mathrm{argmax}}}

\newtheorem{Theorem}{Theorem}
\newtheorem{Remark}{Remark}
\newtheorem{Lemma}{Lemma}
\newtheorem{Proposition}{Proposition}
\newtheorem{Corollary}{Corollary}

\newtheorem{prop}{Proposition}

\journal{Journal of Discrete Applied Mathematics}

\begin{document}
\sloppy
\begin{frontmatter}

\title{Improved Bounds for the Greedy Strategy in  Optimization Problems with Curvature}


\author[mymainaddress]{Yajing Liu\corref{mycorrespondingauthor}}
\cortext[mycorrespondingauthor]{Corresponding author.}
\ead{yajing.liu@ymail.com}
\author[mymainaddress,mysecondaryaddress]{Edwin K. P. Chong}
\ead{Edwin.Chong@Colostate.Edu}
\author[mymainaddress,mysecondaryaddress]{Ali Pezeshki}
\ead{Ali.Pezeshki@Colostate.Edu}
\address[mymainaddress]{Department of Electrical and Computer Engineering, Colorado~State~University, Fort~Collins, CO 80523, USA}
\address[mysecondaryaddress]{Department of Mathematics, Colorado State University, Fort Collins, CO 80523, USA}
\begin{abstract}
Consider the problem of choosing a set of actions to optimize an objective function that is a real-valued polymatroid function subject to matroid constraints. The greedy strategy provides an approximate solution to the optimization problem, and it is known to satisfy some performance bounds in terms of the total curvature.  The total curvature depends on the value of objective function on sets outside the constraint matroid.  If we are given a function defined only on the matroid, the problem still makes sense, but the existing bounds involving the total curvature do not apply. This is puzzling: If the optimization problem is perfectly well defined, why should the bounds no longer apply? This motivates an alternative formulation of such bounding techniques.  The first question that comes to mind is whether it is possible to extend a polymatroid function defined on a matroid to one on the entire power set. This was recently shown to be negative in general. Here, we provide necessary and sufficient conditions for the existence of an \emph{incremental} extension of a polymatroid function defined on the uniform matroid of rank $k$ to one defined on the uniform matroid of rank $k+1$, together with an algorithm for constructing the extension. Whenever a polymatroid objective function defined on a matroid can be extended to the entire power set, the greedy approximation bounds involving the total curvature of the extension apply. However, these bounds still depend on sets outside the constraint matroid. Motivated by this, we define a new notion of curvature called \emph{partial curvature}, involving only sets in the matroid. We derive necessary and sufficient conditions for an extension of the function to have a total curvature that is equal to the partial curvature. Moreover, we prove that the bounds in terms of the partial curvature are in general improved over the previous ones.

To illustrate our results, we first present a task scheduling problem to show that a polymatroid function defined on the matroid can be extended to one defined on the entire power set, and we also derive bounds in terms of the partial curvature, which is demonstrably better than the bound in terms of the total curvature.  As a counterpoint, we then provide an adaptive sensing problem where the total curvature of its extension cannot be made equal to the partial curvature. Nonetheless,  for our specific extension, our result gives rise to a stronger bound.
\end{abstract}

\begin{keyword}
curvature\sep greedy\sep matroid\sep polymatroid \sep submodular
\end{keyword}

\end{frontmatter}


\section{Introduction}
\subsection{Background}
Consider the problem of optimally choosing a set of actions  to maximize an objective function. Let $X$ be a finite ground set of all possible  actions and $f:2^X\rightarrow \mathbb{R}$ be an objective function defined on the power set $2^X$ of $X$. The set function $f$ is said to be a \emph{polymatroid} function \cite{Boros2003} if it is \emph{submodular, monotone}, and $f(\emptyset)=0$ (definitions of being submodular and monotone are given in Section~\ref{sc:II}). Let $\mathcal{I}$ be a non-empty collection of subsets of the ground set $X$. The pair $(X,\mathcal{I})$ is called a \emph{matroid} if $\mathcal{I}$ satisfies the \emph{hereditary} and \emph{augmentation} properties (definitions of hereditary and augmentation are introduced in Section~\ref{sc:II}).  The aim is to find a set in $\mathcal{I}$ to maximize the objective function $f$:
\begin{align}\label{problem}
\begin{array}{l}
\text{maximize} \ \    f(M) \\
\text{subject to} \ \ M\in \mathcal{I}.
\end{array}
\end{align}
The pair $(X,\mathcal{I})$ is said to be a \emph{uniform matroid} of rank $K$ ($K\leq |X|$) when $\mathcal{I}=\{S\subseteq X: |S|\leq K\}$, where $|\cdot|$ denotes cardinality. A uniform matroid is a special matroid, so any result for a matroid constraint also applies to a uniform matroid constraint.

Finding the optimal solution to problem~(\ref{problem}) in general is NP-hard. The \emph{greedy strategy} provides a computationally feasible approach to finding an approximate solution to~(\ref{problem}). It starts with the empty set, and then iteratively adds to the current solution set one element that results in the largest gain in the objective function, while satisfying the matroid constraints. A detailed definition of the greedy strategy is given in Section~\ref{sc:II}. The performance of the greedy strategy has attracted the attention of many researchers, and some key developments will be reviewed in the following section.
\subsection{Review of Previous Work}

Nemhauser \emph{et al.}~\cite{nemhauser1978}, \cite{nemhauser19781} proved that, when $f$ is a polymatroid function, the  greedy strategy yields a \emph{$1/2$-approximation}\footnotemark
% ...
\footnotetext{The term $\beta$-approximation means that $f(G)/f(O)\geq \beta$, where $G$ and $O$ denote a greedy solution and an optimal solution, respectively.} for a general matroid and a $(1-e^{-1})$-approximation for a uniform matroid. By introducing the \emph{total curvature} $c(f)$,\footnotemark
\footnotetext{When there is no ambiguity, we simply write $c$ to denote $c(f)$.} 
$$c(f)=\max_{j\in X^, f(\{j\})\neq f(\emptyset)}\left\{1-\frac{f(X)-f({X\setminus\{j\}})}{f(\{j\})-f(\emptyset)}\right\},$$
Conforti and Cornu{\'e}jols~\cite{conforti1984submodular} showed that, when $f$ is  a polymatroid function, the  greedy strategy achieves   a $1/(1+c)$-approximation for a general matroid and a $(1-(1-{c}/{K})^K)/c$-approximation for a uniform matroid, where $K$ is the rank of the uniform matroid. When $K$ tends to infinity, the bound $(1-(1-{c}/{K})^K)/c$ tends to $(1-e^{-c})/c$ from above. For a polymatroid function,  the total curvature $c$ takes values on the interval $ (0,1]$. In this case, we have $1/(1+c)\geq1/2$ and $(1-(1-{c}/{K})^K)/c > (1-e^{-c})/c \geq (1-e^{-1})$, which implies that the bounds $1/(1+c)$ and $(1-(1-{c}/{K})^K)/c$ are stronger than
the bounds $1/2$ and $(1-e^{-1})$ in \cite{nemhauser1978} and \cite{nemhauser19781}, respectively. Vondr{\'a}k~\cite{vondrak2010submodularity} proved that for a polymatroid function, the continuous greedy strategy gives a $(1-e^{-c})/c$-approximation for any matroid. Sviridenko \emph{et al.}~\cite{sviridenko2015} proved that, a modified continuous greedy strategy gives a $(1-ce^{-1})$-approximation for any matroid, the first improvement over the greedy $(1-e^{-c})/{c}$-approximation of Conforti and Cornu{\'e}jols from~\cite{conforti1984submodular}.

Suppose that the objective function $f$ in problem~(\ref{problem}) is a polymatroid function and the rank of the matroid $(X,\mathcal{I})$ is $K$. By the augmentation property of a matroid and the monotoneity of $f$, any optimal solution can be extended to a set of size $K$. By the definition of the greedy strategy (see Section~\ref{sc:II}), any greedy solution is of size $K$. For the greedy strategy, under a general matroid constraint and a uniform matroid constraint, the performance bounds $1/(1+c)$ and $(1-(1-{c}/{K})^K)/c$ from \cite{conforti1984submodular} are the best so far, respectively, in terms of the total curvature $c$. However, the total curvature $c$, by definition, depends on the function values on sets outside the matroid $(X,\mathcal{I})$.
This gives rise to two possible issues when applying existing bounding results involving the total curvature $c$: 
\begin{itemize}
\item [1.] If we are given a function $f$ defined only on $\mathcal{I}$, then problem~(1) still makes sense, but the total curvature is no longer well defined. This means that the existing results involving the total curvature do not apply. But this surely is puzzling: if the optimization problem (1) is perfectly well defined, why should the bounds no longer apply?
\item [2.] Even if the function $f$ is defined on the entire $2^X$, the fact that the total curvature $c$ involves sets outside the matroid is puzzling. Specifically, if the optimization problem~(1) involves only sets in the matroid, why should the bounding results rely on a quantity $c$ that depends on sets outside the matroid? 
\end{itemize}
The two reasons above motivate us to investigate  more applicable bounds involving only sets in the matroid. 
\subsection{Contributions}
In this paper,  we provide necessary and sufficient conditions for the existence of an {incremental} extension of a polymatroid function defined on the uniform matroid of rank $k$ to one defined on the uniform matroid of rank $k+1$, together with an algorithm for constructing the extension.  Then, it follows that for problem~(\ref{problem}), provided that the polymatroid objective function defined on a matroid can be extended to the entire power set, the greedy strategy satisfies the  bounds $1/(1+d)$ and $(1-(1-{d}/{K})^K)/d$  for a general matroid and a uniform matroid, respectively, where $d=\inf_{g\in \Omega_f} c(g)$ and $\Omega_f$ is the set of all polymatroid functions $g$ on $2^X$ that agree with $f$ on $\mathcal{I}$ (i.e., the set of all extensions of $f$ to $2^X$). When the objective function $f$ is defined on the entire power set, it is clear that $d\leq c(f)$, which implies that the bounds are improved. However, the bounds still depend on sets outside the matroid, because of the way $d$ is defined.

Next, we define a new notion of curvature, called partial curvature and denoted by $b$, involving only sets in the matroid, and we prove that $b(f)\leq c(g)$, where $g$ is any extension of $f$ to the entire power set.  We derive necessary and sufficient conditions for the existence of an extended polymatroid function $g$ such that $c(g)=b(f)$. This gives rise to improved bounds $1/(1+b(f))$ and $(1-(1-{b(f)}/{K})^K)/b(f)$ for a general matroid and a uniform matroid, respectively.

Finally, we present two examples. We first present a task scheduling problem to show that a polymatroid function defined on the matroid can be extended to one defined on the entire power set, and we also derive bounds in terms of the partial curvature, which is demonstrably better than the bound in terms of the total curvature.  As a counterpoint, we then provide an adaptive sensing problem where the total curvature of its extension cannot be made equal to the partial curvature. Nonetheless,  for our specific extension, our result gives rise to a stronger bound.
\subsection{Organization}
In Section~2, we first introduce definitions of polymatroid functions, matroids, and curvature, and then we review performance bounds in terms of the total curvature from \cite{conforti1984submodular}. In Section~\ref{sec3.1},  we prove that any monotone set function defined on the matroid can be extended to one defined on the entire power set and the extended  function can be expressed in a certain  form. In Section~\ref{sec3.2}, we provide necessary and sufficient conditions for the existence of an {incremental} extension of a polymatroid function defined on the uniform matroid of rank $k$ to one defined on the uniform matroid of rank $k+1$. In Section~\ref{sec3.3}, we introduce a particular extension we call the majorizing extension and explore  what kinds of polymatroid  functions can be majorizingly extended to ones defined on the whole power set. In Section~\ref{sec3.4}, we provide an algorithm  for constructing the extension of a polymatroid function defined on a matroid to the entire power set. In Section~\ref{sec4}, we define the partial curvature involving only sets in the matroid and obtain improved bounds in terms of the partial curvature subject to certain necessary and sufficient conditions. In Section~\ref{Examples}, we illustrate our results by considering a task scheduling problem and an adaptive sensing problem. In Section~\ref{openissues}, we conclude with a discussion of open issues. 
\section{Preliminaries}\label{sc:II}
\subsection{Polymatroid Functions and Curvature}
The definitions and terminology in this paragraph are standard (see, e.g., \cite{Edmonds}, \cite{Tutte}), but are included for completeness.
Let $X$ be a finite ground set of actions, and $\mathcal{I}$ be a non-empty collection of subsets of $X$. Given a pair $(X,\mathcal{I})$, the collection $\mathcal{I}$ is said to be \emph{hereditary} if it satisfies property i below and  has the \emph{augmentation} property if it satisfies property ii below:
\begin{itemize}
\item  [i.] (Hereditary) For all $B\in\mathcal{I}$, any set $A\subseteq B$ is also in $\mathcal{I}$.
\item  [ii.] (Augmentation) For any $A,B\in \mathcal{I}$, if $|B|>|A|$, then there exists $j\in B\setminus A$ such that $A\cup\{j\}\in\mathcal{I}$.
\end{itemize}
The pair $(X,\mathcal{I})$ is called a \emph{matroid} if it satisfies both properties i and ii. The pair $(X,\mathcal{I})$ is called a \emph{uniform matroid} when $\mathcal{I}=\{S\subseteq X: |S|\leq K\}$ for a given $K$, called the \emph{rank} of  $(X,\mathcal{I})$. 	In general, the \emph{rank} of a matroid $(X,\mathcal{I})$ is the cardinality of its maximal set. 

Let $2^X$ denote the power set of $X$, and define a set function $f$: $2^X\rightarrow \mathbb{R}$.
The set function $f$ is said to be \emph{monotone} and \emph{submodular} if it satisfies properties~1 and~2 below, respectively:
\begin{itemize}
\item [1.] (Monotone) For any $A\subseteq B\subseteq X$, $f(A)\leq f(B)$.
\item [2.] (Submodular) For any $A\subseteq B\subseteq X$ and $j\in X\setminus B$, $f(A\cup\{j\})-f(A)\geq f(B\cup\{j\})-f(B)$.
\end{itemize}
A set function $f$: $2^X\rightarrow \mathbb{R}$ is called a \emph{polymatroid function} \cite{Boros2003} if it is monotone, submodular, and $f(\emptyset)=0$, where $\emptyset$ denotes the empty set. The submodularity in property~2 means that the additional value accruing from an extra action decreases as the size of the input set increases. This property is also called the \emph{diminishing-return} property in economics.

The \emph{total curvature} \cite{conforti1984submodular} of a set function $f$ is defined as 
\begin{equation}
\label{totalcurvaturedef}
c(f)=\max\limits_{\substack{j\in X \\f(\{j\})\neq f(\emptyset)}}\left\{1-\frac{f(X)-f(X\setminus \{j\})}{f(\{j\})-f(\emptyset)}\right\}.
\end{equation}
 For convenience, we use $c$ to denote $c(f)$ when there is no ambiguity. Note that $0\leq c\leq 1$ when $f$ is a polymatroid function, and $c=0$ if and only if $f$ is  \emph{additive}, i.e., for any set $A\subseteq X$, $f(A)=\sum_{i\in A}f(\{i\})$. When $c=0$, it is easy to check that the greedy strategy coincides with the optimal strategy. So in the rest of the paper, when we assume that $f$ is a polymatroid function, we only consider $c\in(0,1]$.

\subsection{Performance Bounds in Terms of  Total Curvature}
In this section, we review two theorems from \cite{conforti1984submodular}, which bound the performance of the  greedy strategy using the total curvature $c$ for  general matroid constraints and  uniform matroid constraints. We will use these two theorems to derive bounds in Section~\ref{sec4}.

We first define optimal and greedy solutions for \eqref{problem} as follows:

\emph{Optimal solution}: 
A set $O$ is called an optimal solution of (\ref{problem}) if 
\begin{align*}
O\in\mathop{\argmax}\limits_{M\in \mathcal{I}} f(M),
\end{align*}
 where the right-hand side denotes the collection of arguments that maximize $ f(\cdot)$ on $\mathcal{I}$. Note that there may exist more than one optimal solution for problem~(\ref{problem}). When $(X,\mathcal{I})$ is a matroid of rank  $K$, then any optimal solution can be extended to a set of size $K$ because of the augmentation property of the matroid and the monotoneity of the set function $f$.

\emph{Greedy solution}:
A set $G=\{g_1,g_2,\ldots,g_{K}\}$ is called a \emph{greedy} solution of~(\ref{problem}) if
\begin{align*}
g_1&\in\mathop{\argmax}\limits_{\{g\}\in\mathcal{I}} f(\{g\}),
\end{align*} 
and 
for $i=2,\ldots,K,$
\begin{align*}
g_i&\in\mathop{\argmax}\limits_{\substack{g\in{X}\\ \{g_1,\ldots,g_{i-1},g\}\in\mathcal{I}}} f(\{g_1,g_2,\ldots,g_{i-1},g\}).
\end{align*}
Note that there may exist more than one greedy solution for problem~(\ref{problem}). 
\begin{Theorem}\cite{conforti1984submodular}
\label{Theorem2.1}
Let  $(X,\mathcal{I})$ be a matroid and $f$: $2^X\rightarrow \mathbb{R}$ be a polymatroid function with total curvature $c$. Then, any greedy solution $G$ satisfies
$$\frac{f(G)}{f(O)}\geq \frac{1}{1+c}.$$
\end{Theorem}
When $f$ is  a polymatroid function, we have $c\in(0,1]$, and therefore $1/(1+c)\in[1/2,1)$. Theorem \ref{Theorem2.1} applies to any matroid. This means that the bound ${1}/(1+c)$ holds for a uniform matroid  too. Theorem \ref{Theorem2.2} below provides a tighter bound when $(X,\mathcal{I})$ is a uniform matroid.
\begin{Theorem}\cite{conforti1984submodular}
\label{Theorem2.2}
Let $(X,\mathcal{I})$ be a uniform matroid of rank $K$ and  $f$: $2^X\rightarrow \mathbb{R}$ be a polymatroid function with total curvature $c$. Then, any greedy solution $G$ satisfies
\begin{align*}
\frac{f(G)}{f(O)}&\geq\frac{1}{c}\left(1-\left(1-\frac{c}{K}\right)^K\right)> \frac{1}{c}\left(1-e^{-c}\right).
\end{align*}
\end{Theorem}
\vspace{2mm}
The function $(1-(1-{c}/{K})^K)/c$ is nonincreasing in $K$ for $c\in (0,1]$ and $(1-(1-{c}/{K})^K)/c\searrow (1-e^{-c})/c$ when $K\rightarrow \infty$; therefore, $(1-(1-{c}/{K})^K)/c > (1-e^{-c})/c$ when $f$ is a polymatroid function. Also it is easy to check that $(1-e^{-c})/{c} > 1/(1+c)$ for $c\in(0,1]$, which implies that the bound $(1-(1-{c}/{K})^K)/c$ is stronger than the bound $1/(1+c)$ in Theorem \ref{Theorem2.1}.

The bounds in Theorems~\ref{Theorem2.1} and~\ref{Theorem2.2} involve sets not in the matroid, so as stated they do not apply to optimization problems whose objective function is only defined for sets in the matroid. In the following section, we will explore the  extension of polymatroid functions that yield to the bounds in Theorems~\ref{Theorem2.1} and~\ref{Theorem2.2}.

\section{Function Extension}
\label{sec3}
\subsection{Monotone Extension}
\label{sec3.1}
The following proposition states that any monotone set function defined on the matroid $(X,\mathcal{I})$ can be extended to one defined on the entire power set $2^X$, and the extended function can be expressed  in a certain  form.
\begin{prop}
\label{mainprop1}
Let $(X,\mathcal{I})$ be a matroid of rank $K$ and $f:\mathcal{I}\rightarrow\mathbb{R}$ be a monotone set function. Then there exists a monotone set function $g:2^X\rightarrow\mathbb{R}$ satisfying the following conditions:
\begin{itemize}
\item [a.] $g(A)=f(A)$ for all $A\in\mathcal{I}$.
\item [b.] $g$ is monotone on $2^X$.
\end{itemize}
Moreover, any function $g:2^X\rightarrow\mathbb{R}$ satisfying the above two conditions can be expressed as
\begin{equation}
\label{monotoneform}
g(A)=\begin{cases}
f(A), &A\in\mathcal{I},\\
g(B^*)+d_A, & A\notin \mathcal{I},
\end{cases}
\end{equation}
where 
\begin{equation}
\label{Bstar}
B^*\in\argmax\limits_{\substack{B: B\subset A \\|B|=|A|-1}}g(B)
\end{equation}
and $d_A$ is a nonnegative number.
\begin{proof}
Condition~\emph{a} can be satisfied by construction: first set 
\begin{equation}
\label{monotonecon1}
g(A)=f(A)
\end{equation}
 for all $A\in\mathcal{I}$. To prove that there exists a monotone set function $g$ defined on the entire power set $2^X$ satisfying both conditions~\emph{a} and \emph{b}, we prove the following statement by induction:
There exists a set function $g$ of the  form
\begin{equation}
\label{monotoneform2}
g(A)=\begin{cases}
f(A), &A\in\mathcal{I},\\
g(B^*)+d_A, & A\notin \mathcal{I},
\end{cases}
\end{equation}
such that 
$g$ is monotone for sets of size up to $l$ ($l\leq K$), where $B^*$ is given in (\ref{Bstar}) and $d_A$ is a nonnegative number.

First, we prove that the above statement holds for $l=1$. For $g$ to be monotone for sets of size up to 1, it suffices to have that $g(A)\geq 0$ for any set $A\in 2^X$ with $|A|=1$. For $A\in\mathcal{I}$, by (\ref{monotonecon1}) we have that $g(A)=f(A)\geq 0$. For $A\notin\mathcal{I}$, it suffices to set $g(A)=d_{A}$, where $d_A$ is any nonnegaative number. Therefore, the above statement holds for $l=1$.

Assume that the above statement holds for $l=k$. We prove that it also holds for $l=k+1$.  For this, it suffices to prove that for any $A\in 2^X$ with $|A|=k+1$ and any $B\subset A$, we have that $g(A)\geq g(B)$.

Consider any set $A\in\mathcal{I}$ with $|A|=k+1$. By  (\ref{monotonecon1}) we have that $g(A)=f(A)$. For any set $B\subset A$, by the hereditary property of a matroid, we have that  $B\in\mathcal{I}$, which implies that $g(B)=f(B)$. So for any set $A\in\mathcal{I}$ with $|A|=k+1$ and any set $B\subset A$, by the condition that $f$ is monotone on $\mathcal{I}$, we have that $g(A)\geq g(B)$.

Consider any set $A\notin\mathcal{I}$ with $|A|=k+1$. By the induction hypothesis  for $l=k$, we have that for any set $B\subset A$  with $|B|=k$, $g(B)$ is well defined. Set $d_A\geq 0$ and  $$B^*\in\argmax\limits_{\substack{B: B\subset A \\|B|=|A|-1}}g(B),$$ and then define
\[g(A)= g(B^*)+d_A.\]
We have that
\begin{equation}
\label{inductionineq1}
g(A)\geq g(B)
\end{equation}
 for any set $B\subset A$ with $|B|=k$.
For any set $B\subset A$ with $|B|<k$, there must exist a set $A_k$ with $|A_k|=k$ such that $B\subset A_k\subset A$. By the induction hypothesis $l=k$ and (\ref{inductionineq1}), we have that 
\begin{equation}
\label{inductionineq2}
g(A)\geq g(A_k)\geq g(B).
\end{equation}
Combining (\ref{inductionineq1}) and (\ref{inductionineq2}), for any set $A\notin\mathcal{I}$ with $|A|=k+1$ and any set $B\subset A$, we have that $g(A)\geq g(B)$. Therefore, (\ref{monotoneform2})) holds for $l=k+1$.

We have so far shown  that there exists a monotone set function $g:2^X\rightarrow\mathbb{R}$ satisfying conditions~\emph{a} and \emph{b}. Next we prove that any monotone set function $g:2^X\rightarrow\mathbb{R}$ satisfying conditions~\emph{a} and \emph{b} can be expressed as in  (\ref{monotoneform}).

If $g$ satisfies condition~\emph{a}, then we have that 
\begin{equation}
\label{monotoneform3}
g(A)=f(A), \ \forall A\in\mathcal{I}.
\end{equation}
 If $g$ satisfies condition~\emph{b}, then for any set $A\notin\mathcal{I}$, we have that
\[g(A)\geq g(B^*),\]
which implies that there exists a nonnegative number $d_A$ such that 
\begin{equation}
\label{monotoneform4}
g(A)= g(B^*)+d_A,\ \forall A\notin\mathcal{I}.
\end{equation}
Combining~(\ref{monotoneform3}) and (\ref{monotoneform4}), we have that any monotone set function $g:2^X\rightarrow\mathbb{R}$ satisfying conditions~a and b can be expressed by  (\ref{monotoneform}).
\end{proof}
\end{prop}
In Proposition~\ref{mainprop1}, when $A\notin \mathcal{I}$, we define $g(A)$ using $B^*$ as defined in (\ref{Bstar}). But we are not restricted to using $B^*$ as the following lemma shows.
\begin{Lemma}
Assume that $g$ is a monotone set function defined on the uniform matroid of rank $k$. Then, for any set $A$ with  $|A|=k+1$, there exist nonnegative numbers $d_1,d_2,\ldots, d_{M}$ such that $$g(A)=g(A_1)+d_1=g(A_2)+d_2=\cdots=g(A_M)+d_M,$$
where $M = 2^{k+1}-2$ and $A_1,A_2,\ldots, A_M$ denote all nonempty strict subsets of $A$.
\begin{proof}
Without loss of generality, let \[A_M\in\argmax\limits_{\substack{B: B\subset A, |B|=|A|-1 }}g(B).\] By Proposition~\ref{mainprop1}, we have that there exist $d_M\geq 0$ such that $g(A)=g(A_M)+d_M$. Then, for any $i=1,\ldots, M-1$, setting  $d_i=d_M+g(A_M)-g(A_i)$ results in $$g(A)=g(A_1)+d_1=g(A_2)+d_2=\cdots=g(A_M)+d_M,$$
where $d_i\geq 0$, because $d_M\geq 0$ and $g(A_M)\geq g(A_i)$ for $i=1,\cdots, M-1$.
\end{proof}
\end{Lemma}
\subsection{Polymatroid Extension: From Uniform Matroid to Power Set}
\label{sec3.2}
We now turn our attension to extending polymatroid functions. The authors of \cite{Singh2012} pointed out that there are cases where a polymatroid  function defined on a matroid cannot be extended to one that is defined on the entire power set. In the theorem below, we give necessary and sufficient conditions for the existence of an extension of a polymatroid function defined on the uniform matroid of rank $k$ to the uniform matroid of rank $k+1$.

\begin{Theorem}
\label{Extensionthm}
Let  $f:\mathcal{I}\rightarrow \mathbb{R}$ be a polymatroid  function defined on the uniform matroid  of rank $k$.  Then $f$ can be extended to a  polymatroid function $g$ defined on the uniform matroid of rank $k+1$  if and only if  for any  $A\subseteq X$ with $|A|=k+1$, any  $B\subset A$ with $|B|=k
$, and any $a\in B$, 
\begin{equation}
\label{sufneccondition}
f(B)-f(B\setminus\{a\})\geq f(B^*)-f(A\setminus\{a\}),
\end{equation}
where
\begin{equation}
\label{Bstar2}
 B^*\in\argmax\limits_{\substack{B: B\subset A, |B|=k}}f(B).
\end{equation}
\end{Theorem}

\begin{proof}
$\rightarrow$

In this direction, we need to prove that (\ref{sufneccondition}) holds if $g$ is an extended polymatroid function defined on the uniform matroid of rank $k+1$.
If $g$ is a polymatroid function, we have that $g$ is monotone and submodular. If $g$ is monotone, then for any set $A\notin\mathcal{I}$ with $|A|=k+1$, we have 
\begin{equation}
\label{monotonereq1}
g(A)\geq g(B^*).
\end{equation}
If $g$ is submodular, then for  any set $A$ , any set $B\subset A$ with $|B|=k$, and any action $a\in B$, we have 
\begin{equation}
\label{submodularreq2}
g(B)-g(B\setminus\{a\})\geq g(A)-g(A\setminus \{a\}).
\end{equation}
Combining (\ref{monotonereq1}) and (\ref{submodularreq2}), we have 
\begin{equation}
\nonumber
g(B)-g(B\setminus\{a\})\geq g(B^*)-g(A\setminus \{a\}).
\end{equation}
Because $g$ is an extended function of $f$, we have that $g(B)=f(B)$, $g(B^*)=f(B^*)$, $g(B\setminus \{a\})=f(B\setminus \{a\})$, and $g(A\setminus\{a\})=f(A\setminus\{a\})$. Then the above inequality becomes 
\begin{equation}
\nonumber
f(B)-f(B\setminus\{a\})\geq f(B^*)-f(A\setminus \{a\}).
\end{equation}
which means that (\ref{sufneccondition}) holds.

$\leftarrow$

In this direction, we prove that if (\ref{sufneccondition}) holds, then  there exists a polymatroid function $g$ defined on the uniform matroid of rank $k+1$ that agrees with $f$ on the uniform matroid of rank $k$.

By Proposition~\ref{mainprop1}, we have that there exists an extended monotone set function $g$ of the following form defined on the uniform matroid of rank $k+1$:
\begin{equation}
\label{monotoneformuniformmatroid}
g(A)=\begin{cases}
f(A), &|A|\leq k,\\
f(B^*)+d_A, & |A|=k+1,
\end{cases}
\end{equation}
where $B^*$ is defined as in (\ref{Bstar2}) and $d_A$ is nonnegative.

We will prove that there exists $d_A$ for any $A\subset X$ with $|A|=k+1$ such that $g$ defined in (\ref{monotoneformuniformmatroid}) satisfies
 $g(\emptyset)=0$ and $g$ is submodular on $2^X$. 

Because  $f$ is a polymatroid function on the uniform matroid of $k$ and  $g(A)=f(A)$ for any $A\subseteq X$ with $|A|\leq k$, we have that $g(\emptyset)=f(\emptyset)=0$. For $g$ to be submodular on the uniform matroid of rank $k+1$, it suffices to have that for any $A\subseteq X$ with $|A|=k+1$, any $B\subset A$ with $|B|=k$, and any $a\in B$
\begin{equation}
\label{submdular_cond2}
g(B)-g(B\setminus \{a\})\geq g(A)-g(A\setminus \{a\}).
\end{equation}
For any  $A\subseteq X$ with $|A|=k+1$, 
by (\ref{monotoneformuniformmatroid}), we have that $g(A)=f(B^*)+d_A$, where $d_A\geq 0$. The inequality (\ref{sufneccondition}) implies that 
$f(B)-f(B^*)+f(A\setminus\{a\})-f(B\setminus \{a\})\geq 0$. So only if we set $d_A$ to satisfy
\begin{equation}
\label{summodular_eqcon}
  0\leq d_A\leq \min\limits_{B:B\subset A, |B|=k \ \text{and}\ a:a\in B}\{f(B)-f(B^*)+f(A\setminus\{a\})-f(B\setminus \{a\})\},
\end{equation}
we have that $g(A)\leq f(B)-f(B\setminus\{a\})+f(A\setminus\{a\})$, which implies that
(\ref{submdular_cond2}) holds.

This completes the proof.
\end{proof}

\begin{Remark}
Theorem~{\ref{Extensionthm}} provides  necessary and sufficient conditions for the existence of an extension of a polymatroid function defined on the uniform matroid of rank $k$ to the uniform matroid of rank $k+1$. We will show that the function in the following example, taken from~{\cite{Singh2012}}, does not have an extension because  {(\ref{sufneccondition})}  is not satisfied.

Example~1: Let $X=\{1,2,3\}$ and $\mathcal{I}=\{A: A\in X \ \text{and}\ |A|\leq 2\}$. Define $f:\mathcal{I}\rightarrow\mathbb{R}$ as follows:

$f(\emptyset)=0,$

$f(\{1\})=f(\{2\})=f(\{3\})=1$,

$f(\{1,2\})=f(\{1,3\})=1$, and $f(\{2,3\})=2$.

It is easy to show that the above function $f$ is a polymatroid function on the uniform matroid of rank 2. But as we now show, $f$ cannot be extended to a polymatroid function $g$ on the uniform matroid of rank $3$ which is also the power set.

Setting $A=X$, by (\ref{monotoneformuniformmatroid}), it is easy to see that $B^*=\{2,3\}$. Then we have $g(X)=f(\{2,3\})+d_X$, where $d_X\geq 0$.
If  {(\ref{sufneccondition})} holds for $A=X$, $B=\{1,2\}$, and $\{a\}=\{2\}$, we have the following inequality:
$$f(\{1,2\})-f(\{2,3\})+ f(\{1,3\})-f(\{1\})\geq 0.$$
However,
$$f(\{1,2\})-f(\{2,3\})+ f(\{1,3\})-f(\{1\})=-1<0.$$
We conclude that {(\ref{sufneccondition})} does not hold always. Then by Theorem~{\ref{Extensionthm}}, we have that the polymatroid function $f$ defined above does not have an extended polymatroid function defined on the whole power set.
\end{Remark}
\subsection{Majorizing Extension}
\label{sec3.3}
Theorem~\ref{Extensionthm} and Proposition~\ref{mainprop1} together provide us an algorithm to extend a polymatroid function $f$ defined on the uniform matroid of rank $k$ to a polymatroid function $g$ defined on the uniform matroid of rank $k+1$. The procedure is to construct $g$ as in (\ref{monotoneformuniformmatroid}) with $d_A$ satisfying (\ref{summodular_eqcon}).
By (\ref{summodular_eqcon}), if for any $A$ with $|A|=k+1$,
$$\min\limits_{B:B\subset A, |B|=k \ \text{and}\ a:a\in B}\{f(B)-f(B^*)+f(A\setminus\{a\})-f(B\setminus \{a\})\}\geq 0,$$
then $f$ can be extended to $g$.  We say that $f$ is \emph{majorizingly extended} to $g$ if for any $A$ with $|A|=k+1$, we set 
\begin{equation}
\label{majorizingextension}
d_A=\min\limits_{B:B\subset A, |B|=k \ \text{and}\ a:a\in B}\{f(B)-f(B^*)+f(A\setminus\{a\})-f(B\setminus \{a\})\}.
\end{equation}

\begin{Remark}
The reason we are calling this particular construction of $g$ a majorizing extension is that the sequence $\{d_A\}$ (indexed by $A$) majorizes any other sequence $\{d_A'\}$ whose elements satisfy (\ref{summodular_eqcon}), because $d_A\geq d_A'$ for any $A\subseteq X$.

\end{Remark}
We just introduced the definition of a majorizing extension. We wish to explore what kind of polymatroid  functions can be majorizingly extended to ones defined on the whole power set. The following theorem states that a polymatroid function defined on the uniform matroid of rank 1 can be majorizingly extended to one  defined on the power set, and the extended function is additive.
\begin{Theorem}
\label{Polymatroidextension}
Let $X$ be a ground set and $f$  a polymatroid function defined on the uniform matroid of rank 1. Then $f$ can be majorizingly extended to a polymatroid function $g$ defined on the power set $2^X$ with \[g(\{x_1, x_2,\ldots, x_k\})=\sum\limits_{j=1}^kf(\{x_j\})\] for any set $\{x_1,x_2,\ldots, x_k\}\subseteq X$.
\end{Theorem}

\begin{proof}
We will prove the theorem by induction on $k$.
Without loss of generality, we assume for convenience that $X=\{1,2,\ldots, N\}$ and   $f(\{1\})\leq f(\{2\})\leq \cdots\leq f(\{N\})$.

First, we prove the claim for $k=2$, i.e., $g(\{x_1,x_2\})=f(\{x_1\})+f(\{x_2\})$ for any $\{x_1, x_2\}\subseteq X\ (x_1< x_2).$ 
By the assumption above and (\ref{monotoneformuniformmatroid}), we have that $$g(\{x_1, x_2\})=f(\{x_2\})+d_{\{x_1, x_2\}}.$$
By (\ref{majorizingextension}), we have that $d_{\{x_1, x_2\}}= f(\{x_1\}),$ which  results in 
$g(\{x_1, x_2\})=f(\{x_1\})+f(\{x_2\}).$

Now assume that the claim   holds for
$k \leq l$ ($l>2$). Then we prove that it also holds  for 
$k = l+1$ ($l>2$). Without loss of generality, we assume that $x_1<x_2<\cdots <x_{l+1}$.
Then by (\ref{monotoneformuniformmatroid}) and the induction hypothesis for $k\leq l$, we have that \[g(\{x_1, x_2,\ldots, x_{l+1}\})=g(\{x_2, \ldots, x_{l+1}\})+d=\sum\limits_{j=2}^{l+1}f(\{x_j\})+d.\]
For any $B=\{x_1, x_2,\ldots, x_{l+1}\}\setminus\{x_m\}$ and $a=x_n\in B$, by (\ref{majorizingextension}), we have that 
\begin{align*}
d&=  \min\limits_{m,n}
\left\{\sum\limits_{j=1}^{l+1}f(\{x_{j}\})-f(\{x_m\})-\sum\limits_{j=2}^{l+1}f(\{x_j\})+\sum\limits_{j=1}^{l+1}f(\{x_j\})-f(\{x_n\})\right. \notag\\
   &\quad\quad \left.-\left(\sum\limits_{j=1}^{l+1}f(\{x_j\})-f(\{x_m\})-f(\{x_n\})\right)\right\}\\
   &=f(\{x_1\}),
\end{align*}
which results in 
\[g(\{x_1,x_2,\ldots, x_{l+1}\})=\sum\limits_{j=1}^{l+1}f(\{x_j\}).\]
This completes the proof.
\end{proof}
Theorem~\ref{Polymatroidextension} shows that any polymatroid  function defined on the uniform matroid of rank 1 can be majorizingly extended to one defined on the whole power set. The following counterexample shows that the same is not the case for a uniform matroid of rank $2$.

Example~2: Let $X=\{1,2,3,4\}$ and $\mathcal{I}=\{A: A\in X \ \text{and}\ |A|\leq 2\}$. Define $f:\mathcal{I}\rightarrow\mathbb{R}$ as follows:

$f(\emptyset)=0,$

$f(\{1\})=1, f(\{2\})=2, f(\{3\})=3, f(\{4\})=4$,

$f(\{1,2\})=2.0760, f(\{1,3\})=3.2399, f(\{2,3\})=3.3678,$

$f(\{1,4\})=4.1233, f(\{2,4\})=4.4799, f(\{3,4\})=5.2518$.

It is easy to check that $f$ is a polymatroid function on the uniform matroid of rank 2. Now we show that $f$ can not be majorizingly extended to one  on the whole power set. Let $g$ denote the function obtained by the majorizing extension.

By (\ref{monotoneformuniformmatroid}) and (\ref{majorizingextension}), we have that 
$g(\{1,2,3\})=f(\{2,3\})+d_1$, and
\begin{align*}
 d_1&= \min\{f(\{1,2\}-f(\{2\}), f(\{1,2\})-f(\{2,3\}+f(\{1,3\}-f(\{1\}),\\
& \quad\quad\quad \quad f(\{1,3\})-f(\{3\})\}\\
&=0.0760,
\end{align*}
which results in $g(\{1,2,3\})=f(\{2,3\})+d_1=3.4438$.

Similarly, we have that 
$g(\{1,2,4\})=f(\{2,4\})+d_2$, 
$g(\{1,3,4\})=f(\{3,4\})+d_3$,
and $g(\{2,3,4\})=f(\{3,4\})+d_4$, where
\begin{align*}
 d_2&=\min\{f(\{1,2\})-f(\{2\}), f(\{1,2\})-f(\{2,4\})+f(\{1,4\})-f(\{1\}), \\
&\quad\quad\quad\quad f(\{1,4\})-f(\{4\})\}\\
&=0.0760,
\end{align*}
\begin{align*}
d_3&= \min\{f(\{1,3\}-f(\{3\}), f(\{1,3\})-f(\{3,4\})+f(\{1,4\})-f(\{1\}),
\\&\quad\quad\quad \quad f(\{1,4\})-f(\{4\})\}\\
&=0.1233,
\end{align*}
and 
\begin{align*}
d_4&=\min\{f(\{2,3\}-f(\{3\}), f(\{2,3\})-f(\{3,4\})+f(\{2,4\})-f(\{2\}),\\
&\quad\quad\quad\quad f(\{2,4\})-f(\{4\})\}\\
&=0.3678.
\end{align*}
Hence, we have
$g(\{1,2,4\})=f(\{2,4\})+d_2=4.5559$,
$g(\{1,3,4\})=f(\{3,4\})+d_3=5.3751$,
and 
$g(\{2,3,4\})=f(\{3,4\})+d_4=5.6196$.

Now majorizingly construct $g(\{1,2,3,4\})$.
By (\ref{monotoneformuniformmatroid}) and (\ref{majorizingextension}), we have that 
$g(\{1,2,3,4\})=g(\{2,3,4\})+d_5$, and
\begin{align*}
d_5&= \min\{ g(\{1,2,3\})-g(\{2,3,4\}+g(\{1,3,4\})-f(\{1,3\}),
\\& \quad g(\{1,2,3\})-g(\{2,3,4\})+g(\{1,2,4\})-f(\{1,2\}), \\
 &\quad g(\{1,2,4\})-g(\{2,3,4\})+g(\{1,3,4\})-f(\{1,4\}), \\
 &\quad g(\{1,2,3\}-f(\{2,3\}), g(\{1,2,4\})-f(\{2,4\}), g(\{1,3,4\})-f(\{3,4\})\}\\
 &=-0.0406<0.
\end{align*}
Therefore, $g$ defined as above is not  a polymatroid function. However, there are some polymatroid  functions defined on the uniform matroid of rank 2 that can be majorizingly extended to ones defined on the entire power set.
In Section~{\ref{Examples}}, we present two canonical examples that frequently arise in task scheduling and adaptive sensing and show that the objective functions in the two examples can be both majorizingly extended to polymatroid functions defined on the entire power set. Theorem~\ref{Polymatroidextension} implies that any monotone additive function defined on the uniform matroid of rank $k$ ($k>1$) can be majorizingly extended to one defined on the entire power set.
\subsection{Polymatroid Extension: From General Matroid to Power Set}
\label{sec3.4}
Theorem~3   and Proposition~\ref{mainprop1} together provide an iterative algorithm for us to extend  a polymatroid function $f$ defined on the matroid  $(X,\mathcal{I})$ to a polymatroid function $g$ defined on the entire power set. We use $g_k$ to denote a polymatroid function defined on the uniform matroid of rank $k$ satisfying $g_k(A)=f(A)$ for $A\in\mathcal{I}$ with $|A|\leq k$.
The idea is 
that we first define $g_1(A)=f(A)$ for $A\in\mathcal{I}$ with $|A|\leq 1$ and  $g_1(A)\geq 0$ for $A\notin\mathcal{I}$ with $|A|=1$.
 Then, iteratively extend $g_k$ defined on the uniform matroid of rank $k$ to  $g_{k+1}$ defined on the uniform matroid of rank $k+1$ using (\ref{monotoneformuniformmatroid}) and (\ref{summodular_eqcon})  for $k=1, 2,\ldots, |X|-1$. Finally, set $g=g_{|X|}$. This results in 
 \[g_{k+1}(A)=\left\{
                \begin{array}{ll}
                 g_k(A), \ \quad\quad  |A|
              
             \leq k\\
               f(A), \  \quad\quad \ A\in\mathcal{I}  \ \text{with}\ |A|=k+1\\
                  g_k(B^*)+d_{A}, \ \quad A\notin \mathcal{I} \ \text{with}\ |A|=k+1.
                \end{array}
              \right.\]

The specific process is given as follows:
\begin{itemize}
\item  First define \[g_1(A)=\left\{
                \begin{array}{ll}
                  f(A), \  \quad \ A\in\mathcal{I} \ \text{with}\ |A|\leq 1\\
                 d_A, \ \quad A\notin \mathcal{I} \ \text{with}\ |A|=1
                \end{array}
              \right.\]
              where $d_A\geq 0$.
\item Then iteratively define $g_{k+1}(A)$ for $k =1,\ldots,|X|-1$ using the following method:

Assume that $g_k(A)$ is well defined for $|A|\leq k$. For $A\subseteq X$ with $|A|\leq k$, set $g_{k+1}(A)=g_k(A)$. For $A\in\mathcal{I}$ with $|A|=k+1$, set $g_{k+1}(A)=f(A)$. For $A\notin \mathcal{I}$ with $|A|=k+1$, let $B^*\in\argmax\limits_{B:B\subseteq A,|B|=k}g_k(B)$. If 
\[d^*=\min\limits_{a\in B\subseteq A}\left\{[g_k(B)-g_k(B\setminus\{a\})]-[g_k(B^*)-g_k(A\setminus\{a\})]\right\}\geq 0,\]
then set $g_{k+1}(A)=g_k(B^*)+d_{A}$, where $0\leq d_{A}\leq d^*$; else, extension fails.
\item If $g_{|X|}$ exists, set $g=g_{|X|}$.
\end{itemize}

In the algorithm above, we do not specify the exact $d_A$ value. Of course, as before, we can choose $d_A=d^*$, leading to a majorizing extension. As we have seen before, the majorizing extension might not be a polymatroid function even if a polymatroid extension exists. Nonetheless, if indeed  a polymatroid extension exists, then there always exist choices of $d_A$ that produce the extension via the algorithm above. But the problem of finding an appropriate sequence of $d_A$ values can be reduced to that of finding a feasible path in a shortest-path problem (where shortest here could be defined in terms of the smallest total curvature of the extension). Solving this problem is tantamount to solving a problem of the form (1); in general, we would need to resort to something like dynamic programming. This implies that in general, finding a polymatroid function extension is nontrivial.

\section{Improved Bounds}
\label{sec4}

Let $f:2^X\rightarrow \mathbb{R}$ be a polymatroid function. Note that  $f:2^X\rightarrow \mathbb{R}$ is itself an extension of $f$ from $\mathcal{I}$ to the entire $2^X$, and the extended $f:2^X\rightarrow \mathbb{R}$ is a polymatroid function on the entire $2^X$. Therefore, Theorem~\ref{Extensionthm} gives rise in a straightforward way to the following result, stated without proof.
\begin{Proposition}
\label{Proposition1}
Let $(X,\mathcal{I})$ be a matroid and $f:2^X\rightarrow \mathbb{R}$ a polymatroid function on $2^X$. Then $c(f)\geq \inf_{g\in \Omega_f} c(g)$, where $\Omega_f$ is the set of all polymatroid functions $g$ on $2^X$ that agree with $f$ on $\mathcal{I}$.
\end{Proposition}

In this section, we will prove that for problem~$(\ref{problem})$,  if we set $d=\inf_{g\in \Omega_f} c(g)$, then the greedy strategy yields a $1/(1+d)$-approximation and a $(1-e^{-d})/d$-approximation under a general matroid  and a uniform matroid constraint, respectively. Some proofs in this section are straightforward, but are included for completeness.

\begin{Theorem}
\label{Improvedbound}
 Let $(X,\mathcal{I})$ be a matroid of rank $K$ and $f:\mathcal{I}\rightarrow\mathbb{R}$ a polymatroid function.  If there exists an extension of $f$ to the entire power set, then any greedy solution $G$ to problem~$(\ref{problem})$ satisfies
 \begin{equation}
 \label{bound1}
 \frac{f(G)}{f(O)}\geq \frac{1}{1+d},
 \end{equation}
where $d=\inf_{g\in \Omega_f} c(g)$. In particular, when  $(X,\mathcal{I})$ is a uniform matroid, any greedy solution $G$ to problem~$(\ref{problem})$ satisfies
\begin{align}
\label{bound2}
\frac{f(G)}{f(O)}&\geq\frac{1}{d}\left(1-\left(1-\frac{d}{K}\right)^K\right)> \frac{1}{d}\left(1-e^{-d}\right).
\end{align}
\end{Theorem}
\begin{proof}
By Theorems~\ref{Theorem2.1} and~\ref{Theorem2.2}, for any extension $g$ of $f$ to the entire power set, we have the following inequalities 
\[\frac{g(G)}{g(O)}\geq \frac{1}{1+c(g)}\]
and 
\[\frac{g(G)}{g(O)}\geq\frac{1}{c(g)}\left(1-\left(1-\frac{c(g)}{K}\right)^K\right)> \frac{1}{c(g)}\left(1-e^{-c(g)}\right).\]
Because $f$ and $g$ agree on $\mathcal{I}$, we have that $f(G)=g(G)$ and $f(O)=g(O)$. Thus, $(\ref{bound1})$ and~$(\ref{bound2})$ hold for problem~$(\ref{problem})$.
\end{proof}
\begin{Remark}
Because the functions $1/(1+x)$, $(1-(1-x/K)^K)/x$, and $(1-e^{-x})/x$ are all nonincreasing in $x$ for $x\in(0,1]$ and from Proposition~\ref{Proposition1} we have $0<d\leq c(f)\leq 1$ when $f$ is defined on the entire power set, we have that $1/(1+d)\geq 1/(1+c(f))$, $((1-(1-d/K)^K)/d\geq (1-(1-c(f)/K)^K)/c(f)$, and $(1-e^{-d})/d\geq (1-e^{-c(f)})/c(f)$. This implies that our new bounds are, in general, stronger than the previous bounds.
\end{Remark}
\begin{Remark}
The bounds $1/(1+d)$ and $(1-e^{-d})/d$  apply to  problems where the objective function is a polymatroid function defined only for sets in the matroid and can be extended to one defined on the entire power set. However, these bounds still depend on sets not in the matroid, because of the way $d$ is defined.
\end{Remark}
Now we define a notion of \emph{partial curvature} that only involves sets in the matroid.
Let $h:\mathcal{I}\rightarrow \mathbb{R}$ be a set function. We define the partial curvature $b(h)$ as follows:
\begin{equation}
\label{curvature_b}
b(h)=\max_{\substack{j, A: j\in A\in\mathcal{I}\\ h(\{j\})\neq h(\emptyset)}}\left\{1-\frac{h(A)-h(A\setminus\{j\})}{h(\{j\})-h(\emptyset)}\right\}.
\end{equation}

For convenience, we use $b$ to denote $b(h)$ when there is no ambiguity. Note that $0\leq b\leq 1$ when $h$ is a polymatroid function on the matroid $(X,\mathcal{I})$, and $b=0$ if and only if $h$ is additive for sets in $\mathcal{I}$. When $b=0$, the greedy solution to problem~(\ref{problem}) coincides with the optimal solution, so we only consider $b\in(0,1]$ in the rest of the paper. For any extension of $f:\mathcal{I}\rightarrow\mathbb{R}$ to $g:2^X\rightarrow \mathbb{R}$ , we have that $c(g)\geq b(f)$, which will be proved in the following theorem. 

\begin{Theorem}
\label{comparison}
Let $(X,\mathcal{I})$ be a matroid and $f: \mathcal{I}\rightarrow \mathbb{R}$ a polymatroid function. Assume that a polymatroid extension $g:2^X\rightarrow\mathbb{R}$ of $f$ exists. Then $b(f)\leq c(g)$.
\end{Theorem} 
\begin{proof}
By submodularity of $g$ and $g(A)=f(A)$ for any $j\in A\in\mathcal{I}$, we have that
\[f(A)-f(A\setminus \{j\})\geq g(X)-g(X\setminus \{j\}),\]
which implies that for any $j\in A\in\mathcal{I}$,  
\[1-\frac{f(A)-f(A\setminus\{j\})}{f(\{j\})-f(\emptyset)}\leq 1-\frac{g(X)-g(X\setminus\{j\})}{g(\{j\})-g(\emptyset)}.\]
Hence, combining the above with $(\ref{totalcurvaturedef})$ and $(\ref{curvature_b})$ gives $b(f)\leq c(g)$.
\end{proof} 
\begin{Remark}
As mentioned earlier, the improved bounds involving $d$ in Theorem~\ref{Improvedbound} still depend on sets not in the matroid. In contrast, by definition, the partial curvature $b(f)$ depends on sets in the matroid. So if  there exists an extension  of $f$ to $g$ such that $c(g)=b(f)$, then we can derive bounds that are not influenced by sets outside the matroid. However, it turns out that there does not always exist a $g$ such that $c(g)=b(f)$;  we will give an example in Section~4.2 to show this. In the following theorem, we  provide necessary and sufficient conditions for  $c(g)=b(f)$.
\end{Remark}
\begin{Theorem}
\label{iff}
Let $(X,\mathcal{I})$ be a matroid and $f: \mathcal{I}\rightarrow \mathbb{R}$  a polymatroid function. Let $g:2^X\rightarrow\mathbb{R}$ be a polymatroid function that agrees with $f$ on $\mathcal{I}$. Then $c(g)=b(f)$ if and only if 
\begin{equation}
\label{sufneccond2}
g(X)-g(X\setminus \{a\})\geq (1-b(f))g(\{a\})
\end{equation}
 for any $a\in  X$, and equality holds for some $a\in X$.
\end{Theorem}
\begin{proof}
$\rightarrow$

In this direction, we assume that $c(g)=b(f)$ and then to prove  that $g(X)-g(X\setminus \{a\})\geq (1-b(f))g(\{a\})$ for any $a\in  X$ and that equality holds for some $a\in X$. By the definition of the total curvature $c$ of $g$ and $c(g)=b(f)$, we have for any $a\in X$,
\[g(X)-g(X\setminus\{a\})\geq (1-b(f))g(\{a\}),\] 
and equality holds for some $a\in X$.
 
$\leftarrow$

Now we assume that $g(X)-g(X\setminus \{a\})\geq (1-b(f))g(\{a\})$ for any $a\in X$  and that equality holds for some $a\in X$, and then prove that $c(g)=b(f)$. 
By the assumptions, we have
\[1-\frac{g(X)-g({X\setminus\{a\}})}{g(\{a\})-g(\emptyset)}\leq b(f)\]
for any $a\in X$, and equality holds for some $a\in X	$.
By the definition of the total curvature $c$ of $g$, we have
\[c(g)=\max_{\substack{a\in X\\ g(\{a\})\neq g(\emptyset)}}\left\{1-\frac{g(X)-g({X\setminus\{a\}})}{g(\{a\})-g(\emptyset)}\right\}= b(f).\]

This completes the proof. 
\end{proof}
\begin{Remark}
In Section~\ref{Examples}, we will provide a task scheduling example to show that there exists a polymatroid function $g:2^X\rightarrow\mathbb{R}$ that agrees with $f:\mathcal{I}\rightarrow\mathbb{R}$ such that $c(g)=b(f)$. We also provide a contrasting example from an adaptive sensing problem where such an extension does not exist.
\end{Remark}
Combining Theorems~\ref{Improvedbound} and~\ref{iff}, we have the following corollary. 
\begin{Corollary}
\label{ImprovedBoundsCor}
Let $(X,\mathcal{I})$ be a matroid of rank $K$. Let $g:2^X\rightarrow\mathbb{R}$ be a polymatroid function that agrees with $f$ on $\mathcal{I}$ and $g(X)-g(X\setminus \{a\})\geq (1-b(f))g(\{a\})$ for any $a\in X$ with equality holding for some $a\in X$. Then, any greedy solution $G$ to problem~$(\ref{problem})$ satisfies
 \begin{equation}
 \label{bound12}
 \frac{f(G)}{f(O)}\geq \frac{1}{1+b(f)}.
 \end{equation}
 In particular, when  $(X,\mathcal{I})$ is a uniform matroid, any greedy solution $G$ to problem~$(\ref{problem})$ satisfies
\begin{align}
\label{bound22}
\frac{f(G)}{f(O)}&\geq\frac{1}{b(f)}\left(1-\left(1-\frac{b(f)}{K}\right)^K\right)> \frac{1}{b(f)}\left(1-e^{-b(f)}\right).
\end{align}
\end{Corollary}
The bounds ${1}/({1+b(f)})$ and $(1-\left(1-{b(f)}/{K}\right)^K)/b(f)$ do not depend on sets outside the matroid, so they apply to problems where the objective function is only defined on the matroid, provided that an extension that satisfies the assumptions in Theorem~\ref{iff} exists. When $f$ is defined on the entire power set, from Theorem~\ref{comparison}, we have $b(f)\leq c(f)$, which implies that the bounds are stronger than those  from \cite{conforti1984submodular}. 
\section{Examples}
\label{Examples}
We first provide a task scheduling example where we majorizingly extend $f:\mathcal{I}\rightarrow \mathbb{R}$ to a polymatroid function $g_1:2^X\rightarrow\mathbb{R}$ with $c(g_1)>b(f)$. We also extend $f:\mathcal{I}\rightarrow \mathbb{R}$ to another polymatroid function $g_2:2^X\rightarrow\mathbb{R}$ with $c(g_2)=b(f)$. The two extensions both result in stronger bounds than the previous bound from \cite{conforti1984submodular}.
 Then we provide an adaptive sensing example to majorizingly extend $f:\mathcal{I}\rightarrow \mathbb{R}$  to a polymatroid function $g_1:2^X\rightarrow\mathbb{R}$ and show that there does not exist any extension of $f$  to $g$ such that  $c(g)=b(f)$ holds. However, in this example, it turns out that for our majorizing extension $g_1$, $c(g_1)$ is very close to $b(f)$ and is much smaller than $c(f)$. 
\subsection{Task Scheduling}
As a canonical example of problem (\ref{problem}), we will consider the task assignment problem that was posed in \cite{streeter2008online}, and was further analyzed in \cite{ZhC13J}--\cite{YJ2016}.  In this problem, there are $n$ subtasks and a  set $X$ of $N$ agents $a_j$ $(j=1,\ldots, N).$ At each stage, a subtask $i$ is assigned to an agent $a_j$, who successfully accomplishes the task with probability $p_i(a_j)$. Let $X_i({a_1,a_2,\ldots, a_k})$ denote the Bernoulli random variable that describes whether or not subtask $i$ has been accomplished after performing the sequence of actions ${a_1,a_2,\ldots, a_k}$ over $k$ stages. Then $\frac{1}{n}\sum_{i=1}^n	X_i(a_1,a_2,\ldots,a_k)$ is the fraction of subtasks accomplished after $k$ stages by employing agents $a_1,a_2,\ldots, a_k$. The objective function $f$ for this problem is the expected value of this fraction, which can be written as

$$f(\{a_1,\ldots,a_k\})=\frac{1}{n}\sum_{i=1}^n\left(1-\prod_{j=1}^k(1-p_i(a_j))\right).$$

Assume that $p_i(a)>0$ for any $a\in X$; then it is easy to check that $f$ is non-decreasing. Therefore, when $\mathcal{I}=\{S\subseteq X: |S|\leq K\}$, the solution to this problem should be of size $K$.  Also, it is easy to check that the function $f$ has the submodular property.

For convenience, we only consider the special case $n=1$; our analysis can be generalized to any $n\geq 2$. In this case, we have 
\begin{equation}
\label{objectivefunction}
f(\{a_1,\ldots,a_k\})=1-\prod_{j=1}^k(1-p(a_j)),
\end{equation}
where $p(\cdot)=p_1(\cdot)$.

Let $X=\{a_1,a_2,a_3,a_4\}$, $p(a_1)=0.4$, 
$p(a_2)=0.6$, $p(a_3)=0.8$, and $p(a_4)=0.9$. Then, $f(A)$ is defined as in (\ref{objectivefunction}) for any $A=\{a_i,\ldots,a_k\}\subseteq X$. Consider $K=2$, then  $\mathcal{I}=\{S\subseteq X: |S|\leq 2\}$. It is easy to show that $f:\mathcal{I}\rightarrow \mathbb{R}$ is a polymatroid function.

We first majorizingly extend $f:\mathcal{I}\rightarrow \mathbb{R}$ to a polymatroid function $g_1:2^X\rightarrow\mathbb{R}$ using (\ref{monotoneformuniformmatroid}) and (\ref{majorizingextension}).
By (\ref{monotoneformuniformmatroid}), we have that 
$ g_1(\{a_1,a_2,a_3\})=f(\{a_2,a_3\})+d_{\{a_1,a_2,a_3\}},$ $g_1(\{a_1,a_2,a_4\})=f(\{a_2,a_4\})+d_{\{a_1,a_2,a_4\}},$ $g_1(\{a_1,a_3,a_4\})=f(\{a_3,a_4\})+d_{\{a_1,a_3,a_4\}},$ and
 $ g_1(\{a_2,a_3,a_4\})=f(\{a_3,a_4\})+d_{\{a_2,a_3,a_4\}}.$
By  (\ref{majorizingextension}), we have that 
\begin{align*}
 d_{\{a_1,a_2,a_3\}}= \min\{&f(\{1,2\}-f(\{2\}), f(\{1,2\})-f(\{2,3\}+f(\{1,3\}-f(\{1\}),\\
&f(\{1,3\})-f(\{3\})\}= 0.08,
\end{align*}
\begin{align*}
 d_{\{a_1,a_2,a_4\}}=\min\{&f(\{1,2\})-f(\{2\}), f(\{1,2\})-f(\{2,4\})+f(\{1,4\})-f(\{1\}), \\
&f(\{1,4\})-f(\{4\})\}=0.04,
\end{align*}
\begin{align*}
d_{\{a_1,a_3,a_4\}}= \min\{&f(\{1,3\}-f(\{3\}), f(\{1,3\})-f(\{3,4\})+f(\{1,4\})-f(\{1\}),
\\& f(\{1,4\})-f(\{4\})\}=0.04,
\end{align*}
\begin{align*}
d_{\{a_2,a_3,a_4\}}=\min\{&f(\{2,3\}-f(\{3\}), f(\{2,3\})-f(\{3,4\})+f(\{2,4\})-f(\{2\}),\\
&f(\{2,4\})-f(\{4\})\}=0.06.
\end{align*}
Hence,  $g_1(\{a_1,a_2,a_3\})=1, g_1(\{a_1,a_2,a_4\})=1, g_1(\{a_1,a_3,a_4\})=1.02$, and $g_1(\{a_2,a_3,a_4\})=1.04.$

We now construct $g_1(X)$. By (\ref{monotoneformuniformmatroid}), we have that 
$g_1(X)=g_1(\{a_2,a_3,a_4\})+d_X$. By  (\ref{majorizingextension}), we have that
\begin{align*}
d_X= \min\{&g_1(\{1,2,3\})-g_1(\{2,3,4\}+g_1(\{1,3,4\})-f(\{1,3\}),
\\ &g_1(\{1,2,3\})-g_1(\{2,3,4\})+g_1(\{1,2,4\})-f(\{1,2\}),\\
 &g_1(\{1,2,4\})-g_1(\{2,3,4\})+g_1(\{1,3,4\})-f(\{1,4\})\\
 & g_1(\{1,2,3\}-f(\{2,3\}),g_1(\{1,2,4\})-f(\{2,4\}),\\
 &g_1(\{1,3,4\})-f(\{3,4\})\}=0.04,
\end{align*}
Hence, $g_1(X)=g_1(\{a_2,a_3,a_4\})+d_X=1.08$.
Therefore, $g_1$ defined as above is a majorizing extension of $f$ from $\mathcal{I}$ to the whole power set.
 
The total curvature $c$ of $g_1:2^X\rightarrow\mathbb{R}$ is
\begin{align*}
c(g_1)=&\max_{a_i\in X}\left\{1-\frac{g(X)-g(X\setminus\{a_i\})}{g(\{a_i\})-g(\emptyset)}\right\}=0.911
.
\end{align*}
In contrast, the total curvature $c$ of $f$ is 
\begin{align*}
c(f)&=\max_{a_i\in X}\left\{1-\frac{f(X)-f(X\setminus\{a_i\})}{f(\{a_i\})-f(\emptyset)}\right\}\\
&=\max\limits_{a_i,a_j,a_k\in X}\left\{1-\left(1-p(\{a_i\})\right)(1-p(\{a_j\}))(1-p(\{a_k\}))\right\}\\
&=0.992.
\end{align*}
By the definition of the partial curvature $b$ of $f$, we have
\begin{align*}
b(f)&=\max_{\substack{j\in A\subseteq X, |A|=2, \\ f(\{j\})\neq 0}}\left\{1-\frac{f(A)-f(A\setminus\{j\})}{f(\{j\})-f(\emptyset)}\right\}\\
&=\max_{\substack{\{a_i,a_j\}\subseteq X}}\left\{1-\frac{f(\{a_i,a_j\})-f(\{a_i\})}{f(\{a_j\})}\right\}\\
&=\max_{\substack{a_i\in X}}\left\{p(\{a_i\})\right\}=0.9.
\end{align*}

We can see that $c(g_1)$ is close to $b(f)$ and smaller than $c(f)$ though $c(g_1)\neq b(f)$.

Next, we give another extension $g_2$ which satisfies that $c(g_2)=b(f)$. By (\ref{monotoneformuniformmatroid}), we have that
$ g_2(\{a_1,a_2,a_3\})=f(\{a_2,a_3\})+d_{\{a_1,a_2,a_3\}},$ $g_2(\{a_1,a_2,a_4\})=f(\{a_2,a_4\})+d_{\{a_1,a_2,a_4\}},$ $g_2(\{a_1,a_3,a_4\})=f(\{a_3,a_4\})+d_{\{a_1,a_3,a_4\}},$ and
 $ g_2(\{a_2,a_3,a_4\})=f(\{a_3,a_4\})+d_{\{a_2,a_3,a_4\}}.$
First, we will define $d_{\{a_1,a_2,a_3\}}.$
By (\ref{summodular_eqcon}), we have that
\begin{align*}
d_{\{a_1,a_2,a_3\}}\leq \min\{ &f(\{a_1,a_2\})-f(\{a_2\}),f(\{a_1,a_3\})-f(\{a_3\}),\\&f(\{a_1,a_2\})-f(\{a_2,a_3\})+f(\{a_1,a_3\})-f(\{a_1\})\}=0.08.
\end{align*}
By (\ref{sufneccond2}), it suffices to  have that
\begin{align*}
d_{\{a_1,a_2,a_3\}}\geq\max\{&(1-b)f(\{a_1\}),\\
&f(\{a_1,a_3\})-f(\{a_2,a_3\})+(1-b)f(\{a_2\}),\\
&f(\{a_1,a_2\})-f(\{a_2,a_3\})+(1-b)f(\{a_3\})\}=0.04.
\end{align*}

Setting $d_{\{a_1,a_2,a_3\}}= 0.04$ to satisfy the above two inequalities gives that
$g_2(\{a_1,a_2,a_3\})=f(\{a_2,a_3\})+d_{\{a_1,a_2,a_3\}}=0.96$.
Similarly, we set 

$\quad\quad g_2(\{a_1,a_2,a_4\})=f(\{a_2,a_4\})+d_{\{a_1,a_2,a_4\}}=1,$

 $\quad\quad g_2(\{a_1,a_3,a_4\})=f(\{a_3,a_4\})+d_{\{a_1,a_3,a_4\}}=1.02,$
 
 $\quad\quad g_2(\{a_2,a_3,a_4\})=f(\{a_3,a_4\})+d_{\{a_2,a_3,a_4\}}=1.04.$
 
We now define $g_2(X)$. By (\ref{monotoneformuniformmatroid}), we have that $g_2(X)=g_2(\{a_2,a_3,a_4\})+d_X$.
By (\ref{summodular_eqcon}), it suffices to have that
\begin{align*}
d_X&\leq \min\{g_2(\{a_1,a_2,a_4\})-f(\{a_2,a_4\}),\\
&g_2(\{a_1,a_3,a_4\})-f(\{a_3,a_4\}), g_2(\{a_1,a_2,a_3\})-f(\{a_2,a_3\}),\\
&g_2(\{a_1,a_2,a_3\})-g_2(\{a_2,a_3,a_4\})+g_2(\{a_1,a_2,a_4\})-f(\{a_1,a_2\}),\\
&g_2(\{a_1,a_3,a_4\})-g_2(\{a_2,a_3,a_4\})+g_2(\{a_1,a_2,a_4\})-f(\{a_1,a_4\}),\\
&g_2(\{a_1,a_3,a_4\})-g_2(\{a_2,a_3,a_4\})+g_2(\{a_1,a_2,a_3\})-f(\{a_1,a_3\})
\}=0.04.
\end{align*}

By (\ref{sufneccond2}), it suffices to  have that
\begin{align*}
d_X\geq \max\{ &(1-b)f(\{a_1\}),\\
&g_2(\{a_1,a_3,a_4\})-g_2(\{a_2,a_3,a_4\})+(1-b)f(\{a_2\}),\\
&g_2(\{a_1,a_2,a_4\})-g_2(\{a_2,a_3,a_4\})+(1-b)f(\{a_3\}),\\
&g_2(\{a_1,a_2,a_3\})-g_2(\{a_2,a_3,a_4\})+(1-b)f(\{a_4\})
\}=0.04.
\end{align*}

Setting $d_X= 0.04$ to satisfy the above two inequalities gives us $g_2(X)=g_2(\{a_2,a_3,a_4\})+d_X=1.08$.

The total curvature $c$ of $g_2:2^X\rightarrow\mathbb{R}$ is
\begin{align*}
c(g_2)=&\max_{a_i\in X}\left\{1-\frac{g(X)-g(X\setminus\{a_i\})}{g(\{a_i\})-g(\emptyset)}\right\}=0.9=b(f)<c(f)=0.992.
\end{align*}

By Corollary~\ref{ImprovedBoundsCor}, we have that the greedy strategy for the task scheduling problem satisfies the bound $(1-(1-{b(f)}/{2})^2)/b(f)=0.775$, which is better than the previous bound 
$(1-(1-{c(f)}/{2})^2)/c(f)=0.752$.

\subsection{Adaptive Sensing}
For our second example, we consider the adaptive sensing design problem posed in  \cite{ZhC13J}--\cite{YJ2016}. Consider a signal of interest $x \in {\rm I\!R}^2$ with normal prior distribution $\mathcal{N} (0, I)$, where $I$ is the $2\times 2$ identity matrix; our analysis easily generalizes to dimensions larger than $2$. Let $\mathbb{A}=\{\mathrm{Diag}(\sqrt{\alpha},\sqrt{1-\alpha}):  \alpha\in\{\alpha_1,\ldots,\alpha_N\}\}$, where $ \alpha\in[0.5,1]$ for $1\leq i
\leq N$. At each stage $i$, we make a measurement $y_i$ of the form
\[
y_i=a_ix+w_i, 
\]
where $a_i \in\mathbb{A}$ and $w_i$ represents i.i.d.\ Gaussian measurement noise with mean zero and covariance $I$, independent of $x$.

The objective function $f$ for this problem is the information gain \cite{LiC12}, which can be written as
\begin{equation}
\label{objectivefunction2}
f(\{a_1,\ldots, a_k\})=H_0-H_k.
\end{equation}
Here, $H_0=\frac{N}{2}\text{log}(2\pi e)$ is the entropy of the prior
distribution of $x$ and $H_k$ is the entropy of the posterior
distribution of $x$ given $\{y_i\}_{i=1}^k$; that is,
\[
H_k=\frac{1}{2}\text{log det}(P_k)+\frac{N}{2}\text{log}(2\pi e),
\] 
where $P_k=\left(P_{k-1}^{-1}+a_k^Ta_k\right)^{-1}$ is the posterior covariance of $x$ given $\{y_i\}_{i=1}^k$. 

The objective is to choose a set of measurement matrices $\{a_i^*\}_{i=1}^K$, $a_i^*\in\mathbb{A}$, to maximize the information gain $
f(\{a_1,\ldots, a_K\})=H_0-H_K$. It is easy to check that $f$ is monotone, submodular, and $f(\emptyset)=0$; i.e., $f$ is a polymatroid function.

Let $X=\{a_1,a_2,a_3\}$, $\alpha_1=0.5$, 
$\alpha_2=0.6$, and $\alpha_3=0.8$. Then, $f(A)$ is defined as in (\ref{objectivefunction2}) for any $A=\{a_i,\ldots,a_k\}\subseteq X$. Consider $K=2$, where  $\mathcal{I}=\{S\subseteq X: |S|\leq 2\}$.

The total curvature  of $f$ is 
\begin{align*}
c(f)&=\max_{a_i\in X}\left\{1-\frac{f(X)-f(X\setminus\{a_i\})}{f(\{a_i\})-f(\emptyset)}\right\}\\
&=0.4509.
\end{align*}

We first majorizingly extend $f:\mathcal{I}\rightarrow\mathbb{R}$ to a polymatroid function $g_1$ defined on the whole power set. Then we show that there  does not exist a polymatroid extension $g_2$ such that $c(g_2)=b(f)$. However, for the majorizing extension $g_1$, it turns out that $c(g_1)$ is very close to $b(f)$ and is much smaller than $c(f)$.

We start by  majorizingly extending $f$ to $g_1$.
By (\ref{monotoneformuniformmatroid}) and (\ref{majorizingextension}), we have
 $g_1(X)=f(\{a_1,a_2\})+d_X$, where 
 \begin{align*}
d_{X}= \min\{ &f(\{a_1,a_3\})-f(\{a_1\}),f(\{a_2,a_3\})-f(\{a_2\}),\\&f(\{a_1,a_3\})-f(\{a_1,a_2\})+f(\{a_2,a_3\})-f(\{a_3\})\}=\log\sqrt{1.6799}.
\end{align*}
 Hence, $g_1(X)= \log\sqrt{6.7028}$.
 
 The total curvature of $g_1$ is 
\begin{align*}
c(g_1)&=\max_{a_i\in X}\left\{1-\frac{g_1(X)-g_1(X\setminus\{a_i\})}{g_1(\{a_i\})-g_1(\emptyset)}\right\}\\
&=0.3317.
\end{align*}

By the definition of the partial curvature $b$ of $f$, we have
\begin{align*}
b(f)&=\max_{\substack{j\in A\subseteq X, |A|=2, \\ f(\{j\})\neq 0}}\left\{1-\frac{f(A)-f(A\setminus\{j\})}{f(\{j\})-f(\emptyset)}\right\}\\
&=\max_{\substack{\{a_i,a_j\}\subseteq X}}\left\{1-\frac{f(\{a_i,a_j\})-f(\{a_i\})}{f(\{a_j\})}\right\}\\
&=0.3001.
\end{align*}

Comparing the values of $c(g_1), c(f),$ and $b(f)$, we have that $c(g_1)$ is much smaller than $c(f)$ and very close to $b(f)$.
By Theorem~\ref{Improvedbound}, we have that the greedy strategy for the adaptive sensing problem satisfies the bound $(1-(1-{c(g_1)}/{2})^2)/c(g_1)=0.9172$, which is stronger than the previous bound 
$(1-(1-{c(f)}/{2})^2)/c(f)=0.8873$. Now we try to extend $f$ to a polymatroid function $g_2$ such that $c(g_2)=b(f)$.
By (\ref{monotoneformuniformmatroid}),
 $g_2(X)=f(\{a_1,a_2\})+d_X$.
By (\ref{summodular_eqcon}), it suffices to have that
\begin{align*}
d_{X}&\leq \min\{f(\{a_1,a_3\})-f(\{a_1\}),f(\{a_2,a_3\})-f(\{a_2\}),\\&\quad \quad \quad \quad f(\{a_1,a_3\})-f(\{a_1,a_2\})+f(\{a_2,a_3\})-f(\{a_3\})\}\\
&= \log\sqrt{1.6799}.
\end{align*}
By~(\ref{sufneccond2}), it suffices to have that
\begin{align*}
d_{X}&\geq\max\{(1-b(f))f(\{a_3\}),\\
&\quad\quad\quad\quad f(\{a_2,a_3\})-f(\{a_1,a_2\})+(1-b(f))f(\{a_1\}),\\
&\quad\quad\quad\quad f(\{a_1,a_3\})-f(\{a_1,a_2\})+(1-b(f))f(\{a_2\})\},\\
&=\log\sqrt{1.7232}.
\end{align*}

Comparing the above two inequalities, we see that there does not exist $d_{X}$ such that $g_2$ is a polymatroid function satisfying $c(g_2)=b(f)$.

\section{Open Issues}
\label{openissues}
Suppose that a function $f$ defined on a matroid $(X,\mathcal{I})$ is extendable to the entire power set. We have shown that the majorizing extension algorithm does not always successfully produce this extension.  Next, we explored defining a notion of curvature $b(f)$ depending only on sets in the matroid $(X,\mathcal{I})$, and we asked if it is always possible to extend $f$ to $g$ in such a way that $c(g)=b(f)$. Here, again, we have shown that the answer is in general negative; we gave necessary and sufficient conditions for $c(g)=b(f)$. This leaves us with the following ultimate question: What extension $g$ of $f$ has the best (smallest) value of $c(g)$? Unfortunately, answering this question boils down to solving an optimization problem that is in general as difficult as (\ref{problem}), solvable using only  something like dynamic programming. This, of course, does not point to a practical algorithm for finding an extension with the best curvature.

\section{Acknowledgments}
This work is supported in part by NSF under award CCF-1422658, and by the Colorado State University Information Science and Technology Center (ISTeC).

\section*{References}

\bibliography{mybibfile}

\end{document}